\documentclass{proc-l}
\usepackage{graphicx}
\usepackage[all]{xy}
\usepackage[english]{babel}
\usepackage{amssymb,amsmath}
\usepackage{cite}
\usepackage{color}
\usepackage{todonotes}
\usepackage[initials]{amsrefs}

\newtheorem{theorem}{Theorem}[section]
\newtheorem{lemma}[theorem]{Lemma}
\newtheorem{prop}[theorem]{Proposition}
\newtheorem{corollary}[theorem]{Corollary}
\theoremstyle{definition}

\theoremstyle{remark}

\numberwithin{equation}{section}

\renewcommand{\qq}{\mathbb Q}

\newcommand{\zz}{\mathbb{Z}}

\newcommand{\Q}{\mathcal Q}

\newcommand{\GL}{\operatorname{GL}}

\newcommand{\Conceicao}{Concei{\c c}\~ao}

\newcommand{\ovx}{\tilde{x}}
\newcommand{\ovy}{\tilde{y}}
\newcommand{\ovz}{\tilde{z}}
\newcommand{\ovQ}{\tilde{Q}}
\newcommand{\ovf}{\tilde{f}}
\newcommand{\PPT}{SPT}
\newcommand{\SPT}{SPT}

\begin{document}

%    \subjclass is required.
\subjclass[2020]{Primary: 11G35, Secondary: 11T06}


\bibliographystyle{amsplain}

%\title{Pythagorean Berggren Trees for Polynomials}
\title{A polynomial analogue of Berggren's theorem on Pythagorean triples}
%\author[Cha]{Byungchul Cha}
%\author[\Conceicao]{Ricardo \Conceicao}

\author{Byungchul Cha}
\address{Department of Mathematics, 2400 W Chew st., Allentown, PA 18104}
\email{cha@muhlenberg.edu}

\author[\Conceicao]{Ricardo \Conceicao}
\address{Department of Mathematics, 300 N Washington St, Gettysburg, PA 17325}
\email{rconceic@gettysburg.edu}

\commby{}
%\author[Rodriguez]{Miguel Rodriguez}
%\address{Oxford College of Emory University. 100 Hamill st. Oxford, Ga. 30054}
%\email{rconcei@emory.edu}
%\subjclass[2000]{Primary 54C40, 14E20; Secondary 46E25, 20C20}
%\date{March 2012}
\maketitle
\begin{abstract}
  Say that $(x, y, z)$ is a positive primitive integral Pythagorean triple if $x, y, z$ are positive integers without common factors satisfying $x^2 + y^2 = z^2$.
  An old theorem of Berggren gives three integral invertible linear transformations whose semi-group actions on $(3, 4, 5)$ and $(4, 3, 5)$ generate all positive primitive Pythagorean triples in a unique manner.
  We establish an analogue of Berggren's theorem in the context of a one-variable polynomial ring over a field of characteristic $\neq 2$.
  As its corollaries, we obtain some structure theorems regarding the orthogonal group 
  with respect to the Pythagorean quadratic form
  over the polynomial ring.
\end{abstract}

\section{Introduction}
A triple $(x,y,z)\in \zz^3$ is an \emph{integral Pythagorean triple} if it satisfies
\begin{equation}\label{eq:pyt}
 x^2+y^2=z^2.
\end{equation}
%$(x,y,z)$ is said to be \emph{primitive} if $\gcd(x,y,z)=1$. 
It is said to be \emph{primitive} if $\gcd(x,y,z)=1$ and \emph{positive} if $x,y,z>0$.

An old theorem of Berggren \cite{Ber34}, rediscovered independently first by \cite{Bar63} and  later by several other authors\footnote{See the introduction of \cite{Rom08} for a more comprehensive list.}, says that every positive primitive  integral Pythagorean triple can be generated from the well-known integral Pythagorean triple $(3,4,5)$ using four linear transformations, one of which is the permutation of $x$ and $y$. More precisely, if we define
\begin{equation}
  N_1
  =
  \begin{pmatrix}
    1 & -2 & 2 \\
    2 & -1 & 2 \\
    2 & -2 & 3
  \end{pmatrix},
  \quad
  N_2 =
  \begin{pmatrix}
 1 & 2 & 2 \\
 2 & 1 & 2\\
 2 & 2 & 3
 \end{pmatrix},
 \quad
  N_3 =
  \begin{pmatrix}
 -1 & 2 & 2 \\
 -2 & 1 & 2\\
 -2 & 2 & 3
 \end{pmatrix}
 \label{eq:Berggren_matrices}
\end{equation}
then
\begin{theorem}[Berggren's theorem]\label{thm:BerggrenThm}
  Let  $(x, y, z)$ be  a positive primitive  integral Pythagorean triple. Then there exists a unique sequence $\{d_1, \dots, d_k\}\in \{ 1, 2, 3 \}^k$  such that
\[
 \begin{pmatrix}
  x \\ y \\ z
 \end{pmatrix}
 =
 N_{d_1} \cdots N_{d_k}
 \begin{pmatrix}
   3 \\ 4 \\ 5
 \end{pmatrix}
 \text{ or }
 \begin{pmatrix}
  x \\ y \\ z
 \end{pmatrix}
 =
 N_{d_1} \cdots N_{d_k}
 \begin{pmatrix}
   4 \\ 3 \\ 5
 \end{pmatrix}.
\]
(Here, $\{d_1, \dots, d_k\}$ is understood to be an empty sequence if $(x, y, z) = (3, 4, 5)$ or $(4, 3, 5)$.)
\end{theorem}

Berggren's theorem has been generalized to Pythagorean forms in more than two variables \cite{CA90} and certain indefinite binary quadratic forms \cite{CNT}. In both cases, it is shown that all  primitive integral tuples are generated from a finite set of primitive tuples and a finite number of linear transformations. The main result of this paper, Theorem \ref{thm:main_theorem_tree} below, is an   analogue of Berggren's theorem for polynomial rings over a field. Similarly to the results in \cite{CA90} and  \cite{CNT}, it describes how all primitive polynomial Pythagorean triples can be generated from a single polynomial Pythagorean triple using linear transformations and composition of polynomials. The remaining of this introduction is used to make this statement more precise by finding analogues over $K[t]$ of the notions of positive primitive integral Pythagorean triples, the triple $(3,4,5)$ and the matrices in  \eqref{eq:Berggren_matrices}. We also present applications of  the polynomial version of Berggren's Theorem to  the group of linear automorphisms of  \eqref{eq:pyt} over a polynomial ring.

Let $K$ be a field of  characteristic $\neq 2$ and $K[t]$ be the ring of polynomials over $K$ in the indeterminate $t$. A non-zero triple $(x,y,z)\in K[t]^3$  satisfying \eqref{eq:pyt}  is called a \emph{(polynomial) Pythagorean triple}. As before,  $(x,y,z)$ is said to be primitive if $\gcd(x,y,z)=1$\footnote{Recall that $\gcd(x, y, z)$ is by definition the unique monic polynomial in $K[t]$ that generates the smallest ideal of $K[t]$ containing $x, y, z$.}.  The analogue of a positive primitive integral Pythagorean triple is a primitive Pythagorean triple $(x, y, z) \in K[t]^3$ such that $\deg x<\deg y = \deg z$ and the leading coefficients of $y$ and $z$ are the same. We call them \emph{standard Pythagorean triples}  or \SPT\ in short.
As in the classical case, any non-standard primitive Pythagorean triple can be brought to a standard one by means of a $K$-linear coordinate change (cf.~Lemma~\ref{lem:non_SPT_Rf}). Moreover, it follows from the primality condition that all \SPT's $(x,y,z)$ with $x=0$ are of the form 
$(0,c,c)$,
for some $c\in K^*$. Therefore, we restrict ourselves to the study of \SPT's with $x\neq0$.

To find an analogue to $(3,4,5)$, we observe that 
$$
 S_t=(2t,t^2-1,t^2+1),
$$
which comes from the classical  rational parametrization of the unit circle,
yields an SPT of smallest height (see  definition of height in \S\ref{sec:Prelim}). But, because \eqref{eq:pyt} is defined over $K$, new \SPT's can be created from other \SPT's by replacing $t$  with a polynomial $f\in K[t]\backslash K$. In particular,  
\begin{equation}\label{eq:St}
S_f=(2f,f^2-1,f^2+1)
 \end{equation}
 is also an \SPT, which we take to be the natural analogue over $K[t]$ of the triple $(3,4,5)$. 

As for the   linear transformations \eqref{eq:Berggren_matrices} appearing in Berggren's theorem, in \S\ref{sec:Prelim}
we explain how they can be constructed from reflections on a quadratic space defined by 
the quadratic form $\Q(x, y, z) = x^2 + y^2 - z^2 $ associated to \eqref{eq:pyt}. When this construction is applied to reflections defined by a polynomial $f\in K[t]$, we arrive at the 
matrix
$$
 M_f=\left(\begin{array}{rrr}
-1 & 2 f & 2 f \\
-2 f & 2 f^{2} - 1 & 2 f^{2} \\
-2 f & 2 f^{2} & 2 f^{2} + 1
\end{array}\right).
$$

%and $M_f$ appear naturally as reflections on a quadratic space defined by 
%the quadratic form
%\begin{equation}\label{pyth_quadform}
%\Q(x, y, z) = x^2 + y^2 - z^2 
%\end{equation} associated to \eqref{eq:pyt}.
%
%
%In the following polynomial analogue of Berggren's theorem (to be proved in  \S\ref{sec:Tree_structure_SPT}), the role of the matrices \eqref{eq:Berggren_matrices} is played by 
%$$
% M_f=\left(\begin{array}{rrr}
%-1 & 2 f & 2 f \\
%-2 f & 2 f^{2} - 1 & 2 f^{2} \\
%-2 f & 2 f^{2} & 2 f^{2} + 1
%\end{array}\right),
%$$
%for some $f\in K[t]$.

This is the final piece needed to state the following analogue of Berggren's theorem, which is proved in  \S\ref{sec:Tree_structure_SPT}.
\begin{theorem}\label{thm:main_theorem_tree}
Let $Q=(x,y,z)$ be an SPT with $x\neq 0$. 
Then there exist  $c\in K^*$, $f\in K[t]\backslash K$, and a (possibly empty) sequence  $\{f_1, \dots f_k \}$ in $K[t] \backslash K$ such that
\[
Q^T = cM_{f_1} \cdots M_{f_k} S_f^T.
\]
Moreover, this representation of $Q$ is unique.
\end{theorem}

%In  \S\ref{sec:Prelim}, we explain how the linear transformations in \eqref{eq:Berggren_matrices} and $M_f$ appear naturally as reflections on a quadratic space defined by 
%the quadratic form
%\begin{equation}\label{pyth_quadform}
%\Q(x, y, z) = x^2 + y^2 - z^2 
%\end{equation} associated to \eqref{eq:pyt}.

As consequences of this polynomial Berggren's theorem, we obtain the following two  theorems on the orthogonal group $O_{\Q}(K[t])$ of the quadratic form $\Q$.
%\todo[inline]{BC: I realize that the (old) Theorem 1.2 (about the stabilizer) is not really a corollary to Theorem 1.1; it is proved directly by Conrad's argument. So I moved it to \S3 and called it a proposition. Instead, I put in another theorem on the transitive action of $O_{\Q}$ for all Pythagorean triples.}

\begin{theorem}\label{thm:transitivity_action}
 The group $O_{\Q}(K[t])$ acts transitively on the set of all primitive Pythagorean triples.
\end{theorem}
For each $c\in K^*$ and $f\in K[t]$, define 
\begin{equation}
\label{eq:definition_T_c}
T_c =
      \begin{pmatrix}
       1 & 0 & 0 \\
       0 & (c + c^{-1})/2 &  (c - c^{-1})/2  \\
       0 & (c - c^{-1})/2 &  (c + c^{-1})/2   \\
      \end{pmatrix}
\end{equation}
and
\begin{equation*}%\label{EqDefRt}
R_f= \left(\begin{array}{rrr}
-1 & -2 f & 2 f \\
-2 f & -2 f^{2} + 1 & 2 f^{2} \\
-2 f & -2 f^{2} & 2 f^{2} + 1
\end{array}\right).
\end{equation*}
It is straightforward to see that $T_c$ and $R_f$ preserve the form $\Q(x, y, z)$.
\begin{theorem}\label{thm:orthogonal_generator}
  The group $O_{\Q}(K[t])$ is generated by the following set:
    \[
      \{ R_f \mid f \in K[t] \} 
      \cup 
      \left\{
P_{xy}
      \right\}
      \cup
      \left\{ 
      T_c \mid c\in K^*
      \right\} 
    \]
    where $P_{xy}$ is the permutation $(x, y, z) \mapsto (y, x, z)$.
\end{theorem}

The paper is organized as follows. In \S\ref{sec:Prelim} we present some preliminary definitions and results that will be used in the proofs of the three theorems above. The proof of Theorem~\ref{thm:main_theorem_tree} is given in
 \S\ref{sec:Tree_structure_SPT}, while the proofs of Theorems~\ref{thm:transitivity_action} and \ref{thm:orthogonal_generator} are given in \S\ref{sec:Corolaries}.

\section{Definitions and preliminary results}\label{sec:Prelim}
%The following collection of definitions and preliminaries results  are used throughout the text.

Recall that $K$ is a field of  characteristic $\neq 2$. Given a non-zero polynomial $f\in K[t]$, we denote by $\ell(f)$ the (nonzero) leading coefficient of $f$ and $\deg(f)$ the degree of $f$. 
We adopt the convention that $\deg(0) = -\infty$.
For a triple $Q = (x, y, z)\in K[t]^3$,  the \emph{height of $Q$} is the integer $$
h(Q)=\max\{\deg(x), \deg(y), \deg(z)\}.
$$

When $\gcd(x, y, z) = 1$, we say that $Q$ is \emph{primitive}. Given a ring $R$, recall that $A \in \GL_3(R)$  if both $A$ and $A^{-1}$ are defined over $R$. Below we record two properties of height and primitivity of triples that will be useful later.

\begin{lemma}  \label{lem:preserve_primitivity}
For a triple $Q = (x, y, z)\in K[t]^3$ and $A \in \GL_3(K[t])$, write $\ovQ^T  = (\ovx, \ovy, \ovz)^T = AQ^T$.
  Then $\gcd(x, y, z) = \gcd(\ovx, \ovy, \ovz)$.
  In particular, $Q$ is primitive if and only if $\ovQ$ is primitive.
\end{lemma}
\begin{proof}
Write
$f = \gcd(x, y, z)$ and $\ovf = \gcd(\ovx, \ovy, \ovz)$.
Then $f$ divides any $K[t]$-linear combination of $x, y, z$. Therefore, $f$ divides each of $\ovx$, $\ovy$, $\ovz$, thus $\ovf$ as well.
Apply the same argument with $A^{-1}$ to show that $\ovf$ divides $f$.
So we conclude that $f = \ovf$, which clearly implies that $A$ preserves primitivity.
\end{proof}
\begin{lemma}\label{lem:map_over_K_preserve_height}
For a triple $Q = (x, y, z)\in K[t]^3$ and $A \in \GL_3(K)$,
 $$
 h(Q) = h(AQ^T).
 $$
\end{lemma}
\begin{proof}
  As before, write $Q = (x, y, z)$ and $\ovQ^T  = (\ovx, \ovy, \ovz)^T = AQ^T$.
  Then the degree of any $K$-linear combination of $x, y, z$ cannot exceed $h(Q)$, therefore $h(\ovQ) \le h(Q)$. 
  Apply the same argument with $A^{-1}$, which would give $h(Q) \le h(\ovQ)$.
  This completes the proof.
\end{proof}

To find analogues over $K[t]$ of the matrices $N_1, N_2, N_3$  in \eqref{eq:Berggren_matrices}, 
we first contextualize their construction using a geometric interpretation that first appeared in Conrad's note \cite{Conrad} and that was later generalized by \cite{CNT}.  Since this construction works simultaneously for both $\zz$ and $K[t]$, we briefly consider the more general framework of  an integral domain $D$ of characteristic $\neq 2$ and its fraction field $F$. 

We view  the quadratic form 
\begin{equation}\label{eq:pyth_quadform}
\Q(x, y, z) = x^2 + y^2 - z^2 
\end{equation}
associated to \eqref{eq:pyt} as being defined over $D$. The orthogonal group $O_{\Q}(D)$ of $\Q$ over $D$ is, by definition, the group of matrices $A\in GL_3(D)$ satisfying $\Q(A\mathbf{x}) = \Q(\mathbf{x})$, for all $\mathbf{x}\in D^3$. 

Note that $\Q$ defines the bilinear pairing $\langle , \rangle: F^3 \times F^3 \longrightarrow F$
\begin{equation*}
  \label{eq:Berggren_bilinear_pairing}
  \langle \mathbf{x}, \mathbf{y} \rangle = 
  \frac{1}{2}
  \left( 
    \Q(\mathbf{x} + \mathbf{y})  - \Q(\mathbf{x}) - \Q(\mathbf{y})
\right),
\end{equation*}
 for $\mathbf{x}, \mathbf{y} \in F^3$. For $\mathbf{w}\in F^3$ with $\Q(\mathbf{w}) \neq 0$, we define a \emph{reflection} $R_{\mathbf{w}}$ with respect to $\mathbf{w}$  to be the linear map from $F^3$ onto itself given by
\begin{equation}
  R_{\mathbf{w}}(\mathbf{x})
  =
  \mathbf{x}
  -
  2\frac{ \langle \mathbf{x}, \mathbf{w}\rangle}{\Q(\mathbf{w})}\mathbf{w}.
  \label{eq:Berggren_reflection_definition}
\end{equation}
The map $R_{\mathbf{w}}$  is easily seen to have order 2 and to be an element of $O_{\Q}(F)$. If $\mathbf{w} \in F^3$ is such that $\Q(\mathbf{w})=\pm1,\pm2$, then $R_{\mathbf{w}}$ is actually an element of $O_{\Q}(D)$.

When we specialize to the case $D=\zz$, $F=\qq$, $\mathbf{w}=(1,1,1)$,  we may regard $R_{\mathbf{w}}$ as a $3\times 3$ matrix with respect to the standard basis of $\qq^3$. Under this point of view, for $d = 1, 2, 3$,
the matrices $N_d$ in \eqref{eq:Berggren_matrices}   are given by 
\begin{equation*}
N_d = R_{\mathbf{w}} U_d
  \label{eq:Berggren_construction_Md}
\end{equation*}
 where $U_1, U_2, U_3$ are defined by
\begin{equation}
  U_1
  =
  \begin{pmatrix}
    1 & 0 & 0 \\
    0 & -1 & 0 \\
    0 & 0 & 1
  \end{pmatrix},
  \quad
  U_2 =
  \begin{pmatrix}
 -1 & 0 & 0 \\
 0 & -1 & 0\\
 0 & 0 & 1
 \end{pmatrix},
 \quad
  U_3 =
  \begin{pmatrix}
 -1 & 0 & 0 \\
 0 & 1 & 0 \\
 0 & 0 & 1
 \end{pmatrix}.
  \label{eq:def_Berggren_Ud}
\end{equation}

Considering the case where $D=K[t]$, $F=K(t)$ 
%\todo{BC: Does the notation $K((t))$ indicate the field of fraction of $D$ (the rational function field) or the field of Laurent series? 
%I wonder if you mean to say the first (the field of fraction) but I am familiar with the notation $K(\!(t)\!)$ being used for the field of Laurent series.}
and $\mathbf{w}=(1,f,f)$, for some $f\in K[t]$, we denote by $R_f$ the reflection with respect to $\mathbf{w}$ using the formula \eqref{eq:Berggren_reflection_definition}.
The matrix representation of $R_f$ with respect to the standard basis of $K(t)^3$ is
\begin{equation}\label{EqDefRt}
R_f= \left(\begin{array}{rrr}
-1 & -2 f & 2 f \\
-2 f & -2 f^{2} + 1 & 2 f^{2} \\
-2 f & -2 f^{2} & 2 f^{2} + 1
\end{array}\right).
\end{equation}
Observe that $R_f$ is the matrix appearing in Theorem \ref{thm:orthogonal_generator}, while
the matrix  $M_f$ appearing in the statement of Theorem \ref{thm:main_theorem_tree} is given by the product
%(cf.~\eqref{eq:def_Berggren_Ud})
\begin{equation}  \label{eq:definition_M_f}
 R_fU_1=M_f=\left(\begin{array}{rrr}
-1 & 2 f & 2 f \\
-2 f & 2 f^{2} - 1 & 2 f^{2} \\
-2 f & 2 f^{2} & 2 f^{2} + 1
\end{array}\right).
\end{equation}

Since $\Q(1,f, f) = 1$, we see that both $R_f$ and $M_f$ are elements of $O_{\Q}(K[t])$, with $R_f$ having order 2.
We record here $M_f^{-1}$ for future use:
\begin{equation}\label{eq:rmininv}
M_f^{-1}= U_1
  R_f
  =
\left(\begin{array}{rrr}
-1 & -2 f & 2 f \\
2 f & 2 f^{2} - 1 & -2 f^{2} \\
-2 f & -2 f^{2} & 2 f^{2} + 1
\end{array}\right).
\end{equation}

The matrix $R_f$ also satisfies the following identities. 
\begin{lemma}\label{prop:r+identity}
  For $a, b\in K[t]$, we have
  \[
    \begin{cases}
      R_{a}R_{b} = R_{a - b}R_0,\\
    R_aR_0R_b = R_{a + b}.
  \end{cases}
  \]
\end{lemma}
\begin{proof}
  Both of these equations are easily verified by direct computation. 
\end{proof}

  As defined in the introduction, a \emph{standard Pythagorean triple (\SPT)} is a primitive Pythagorean triple $(x,y,z)$ satisfying  $\deg x<\deg y=\deg z$ and $\ell(y) =\ell(z)$. \SPT's play the role over $K[t]$  of a positive primitive integral Pythagorean triple. In the classical setting, every   primitive integral Pythagorean triple can be obtained from a positive one by an element of $O_{\Q}(\zz)$; namely, a change of sign. The  next two results are used to show that, similar to the integral case, any non-standard primitive Pythagorean triple can be obtained from \SPT's via  multiplication by an element of $O_{\Q}(K[t])$. %Our first observation on the non-standard Pythagorean triples is that they can be grouped into four categories.

\begin{lemma}\label{lem:non_SPT}
Suppose that $Q = (x, y, z) \in K[t]^3$ is a Pythagorean triple but not an \SPT. 
Then $Q$ is one of the following types:
 \begin{itemize}
 \item[{(I)}] $\deg x<\deg y=\deg z$ and $\ell(y)=-\ell(z)$.
 \item[{(II)}] $\deg y<\deg x=\deg z$.
 \item[{(III)}] $\deg x=\deg y=\deg z$.
 \item[{(IV)}] $\deg z < \deg x = \deg y$. In this case, $K$ must contain a square root of $-1$, say, $i = \sqrt{-1}$, and  $\ell(x)=\pm i\ell(y)$.
\end{itemize}
\end{lemma}
\begin{proof}
The proof follows easily by comparing the degrees and leading coefficients of the polynomials appearing in both sides of the equation \eqref{eq:pyt}. We leave the details to the reader.
\end{proof}

\begin{lemma}\label{lem:NonPPTtoPPT}
If $Q$ is a primitive Pythagorean triple that is not an \PPT\  and $f\in K[t]\backslash K$  then $R_fQ$ is an \PPT.
\end{lemma}
\begin{proof}
 Write $Q=(x,y,z)$. 
 Then $Q$ is one of the type (I)--(IV) in Lemma~\ref{lem:non_SPT}.
 Let $d = z - y$.
 We claim that $d$ is a non-zero polynomial with $\deg d = \max\{\deg y,\deg z\} = h(Q)$.
 This claim is obvious if $\deg z \neq \deg y$ (type (II) or (IV)). 
 Also, if $Q$ is type (I), then clearly $\deg d = \deg y = \deg z$. 
 In case (III), we must have $\ell(x)^2+\ell(y)^2=\ell(z)^2$, which implies $\ell(y)\neq\pm \ell(z)$ (because $\ell(x)$ is nonzero). 
 Therefore, the claim is proved.
 
If $\begin{pmatrix}
 \ovx,\ovy,\ovz
\end{pmatrix}^T=R_f
\begin{pmatrix}
 x,y,z
\end{pmatrix}^T$ then \eqref{EqDefRt} gives
\begin{equation}\label{eq:rplus_of_f}
\begin{aligned}
 \ovx  &=   -x +2 f d,\\
 \ovy&=-2 f x  +2 f^{2}d +  y=(\ovx-x)f+y, \\
 \ovz&=-2 f x  +2 f^{2}d +  z=(\ovx-x)f+z.
\end{aligned}
 \end{equation}
As a consequence,
$$
\deg \ovx=\deg d+\deg f<\deg d+2\deg f=\deg \ovy=\deg \ovz
$$
and  $\ell(\ovy)=\ell(\ovz)=\ell(2f^2d)$. 
%To finish our proof, we need to show that $\gcd(\ovx,\ovy,\ovz)=1$. For that matter, let $\gcd(\ovx,\ovy,\ovz)=g$. Notice that $g$ divides $d$ since $\ovz-\ovy=d$. Therefore $g$ divides $x=2fd-\ovx$. Also, since $y=\ovy-(\ovx-x)f$ and $z=\ovz-(\ovx-x)f$, we see that $g$ divides both $y$ and $z$. Thus $g$ divides $\gcd(x,y,z)=1$, and  $g=1$ as desired.
\end{proof}

We end this section of preliminary results with a lemma relating the height of an SPT $Q$ with that of $M_fQ^T$, for some $f\in K[t]\backslash K$.
\begin{lemma}\label{lem:RfPPT}
 If $Q$ is an \PPT\ then, for all $f\in K[t]\backslash K$, $M_fQ^T=(\ovx,\ovy,\ovz)^T$ is an \PPT\ with $\ovx\neq 0$.
 Furthermore,
 \[
h(M_fQ^T) = 2\deg f + h(Q). 
 \]
\end{lemma}
\begin{proof}
 Write $Q=(x,y,z)$ and let $\begin{pmatrix}
 \ovx,\ovy,\ovz
\end{pmatrix}^T=M_f
\begin{pmatrix}
 x,y,z
\end{pmatrix}^T$. Then \eqref{eq:definition_M_f} implies
\begin{equation}\label{eq:rminus_of_f}
\begin{array}{ccl}
 \ovx  &=&   -x +2 f (y+z)\\
 \ovy&=&-2 f x  +2 f^{2}(y+z) -  y \\
 \ovz&=&-2 f x  +2 f^{2}(y+z) +  z
\end{array}
 \end{equation}
 Because $Q$ is an \PPT\ and $f$ is non-constant, by analyzing degrees we can see that the leading terms from both $\ovy$ and $\ovz$ come from the polynomial $2 f^{2}(y+z)$, while the leading term from $\ovx$ comes from $2 f (y+z)$. From that, it follows easily that $M_fQ^T$ is  an \PPT\ with $\ovx\neq 0$ and $h(M_fQ^T) = 2\deg f + h(Q)$.
\end{proof}
\section{Proof of Theorem \ref{thm:main_theorem_tree}}
\label{sec:Tree_structure_SPT}

In this section, we use infinite descent to prove the existence of a product representation of an \SPT\ as given in  Theorem \ref{thm:main_theorem_tree}. Its proof in Corollary \ref{cor:infinite_desc}  is a consequence of the next two propositions. The proof of the uniqueness of the representation in Theorem \ref{thm:main_theorem_tree} is found in Corollary \ref{cor:uniqueness}.

\begin{prop}\label{prop:descent}
  Let $Q=(x,y,z)$ be an \PPT\ with $x \neq 0$ and $\deg z\neq 2\deg x$. Then there exists $f\in K[t]\backslash K$ such that
$\ovQ^T=\begin{pmatrix}
 \ovx,\ovy,\ovz
\end{pmatrix}^T =M_f^{-1}Q^T$ is an \PPT\ with $\ovx\neq0$ and $h(\ovQ)<h(Q)$.
%\begin{itemize}
% \item If $\deg z= 2\deg x$ then 
%  $\ovQ=(0,c,c)$, for some $c\in K^*$.
% \item If $\deg z\neq 2\deg x$ then 
%  $\ovQ$ is a \PPT\ with $\ovx\neq0$ and $h(\ovQ)<h(Q)$.
%\end{itemize}
\end{prop}
\begin{proof}
 %Write $z=fx+k$ where $f,k\in K[t]$ and $\deg k<\deg x$. Our goal is to show that 
%$\begin{pmatrix}
% \ovx,\ovy,\ovz
%\end{pmatrix}^T=M_f^{-1}
%\begin{pmatrix}
% x,y,z
%\end{pmatrix}^T$
%is such that $\deg \ovx<\deg \ovy=\deg \ovz$,  $\ovy$ and $\ovz$ have the same leading coefficients and $\deg \ovz<\deg z$.
Let $f\in K[t]$ satisfy $z=fx+k$ with $\deg k<\deg x$. Notice that $f\not\in K$ since $\deg z>\deg x$.

From \eqref{eq:rmininv}, it follows that
\begin{eqnarray*}
 \ovx  &=&   -x -2 f y + 2 f z\\
 \ovy&=&2 f x + (2 f^{2} - 1) y -2 f^{2} z\\
 \ovz&=&-2 f x  -2 f^{2} y + (2 f^{2} + 1) z.
\end{eqnarray*}
If we let $d=z-y$ then the above expressions can be simplified into
\begin{eqnarray}
\label{eq:bx}\ovx &=& -x+2fd\\
\label{eq:by}\ovy&=&-\ovz+d
\end{eqnarray}
Moreover, since $(\ovx,\ovy,\ovz)$ is a Pythagorean triple, we have that
\begin{equation}\label{eq:barxD}
\ovx^2+d^2=2\ovz d. 
\end{equation}
 We also observe that \eqref{eq:pyt} yields
\begin{equation}\label{eq:pytD}
x^2=d(z+y). 
\end{equation}
From this last equality,  the definition of $f$ and $k$, and \eqref{eq:bx} we obtain
\begin{equation}\label{eq:bxDk}
\ovx x=-d(z+y)+2(z-k)d=d(d-2k). 
\end{equation}

Because $(x,y,z)$ is an \PPT\ with $x\neq 0$,  we conclude that $d\neq 0$ and, by \eqref{eq:pytD},
\begin{equation}\label{eq:degD}
\deg d= 2\deg x-\deg z.
\end{equation}

%If we assume that $2\deg x=\deg z$ then $d\in K^*$. Consequently,   if $\ovx\neq 0$ then \eqref{eq:bxDk} would yield the contradictory statement $$
%\deg \ovx+\deg x=\deg(d- 2k)<\deg x.
%$$
%Therefore, $\ovx=0$, and \eqref{eq:barxD} and  \eqref{eq:by} show that $\ovz=d/2$ and $\ovy=d/2$. The result follows by taking $c=d/2$.

Since $(x,y,z)$ is an \PPT\ with $\deg z\neq 2\deg x$ then
 \eqref{eq:degD} and $\deg x<\deg z$ show that  
 \begin{equation}\label{eq:degdx}
0<\deg d<\deg x.     
 \end{equation}
Thus, $\deg(d-2k)<\deg x$ and, by equating degrees in \eqref{eq:bxDk}, we arrive at
$$
\deg \ovx+\deg x=\deg d+\deg(d-2k)<\deg d+\deg x.
$$
This shows that $\deg \ovx<\deg d$. Consequently,  \eqref{eq:barxD} proves that $\deg \ovz=\deg d$ and  $2\ell(\ovz)=\ell(d)$. Together with \eqref{eq:by}, we conclude that $\ell(\ovy)=\ell(\ovz)$ and that $\deg \ovy=\deg \ovz=\deg d>\deg \ovx$, proving  that $(\ovx,\ovy,\ovz)$ is an \PPT. 

If we assume that $\ovx=0$, then $\ovz=\ovy$ and 
$$
\ovz=\gcd(\ovx,\ovy,\ovz)=\gcd(x,y,z)=1.
$$ 
Moreover, \eqref{eq:barxD} shows that $1=\ovz=d/2$. Since this equality contradicts \eqref{eq:degdx}, we conclude that $\ovx\neq 0$. To finish our proof, notice that $Q^T=M_f\ovQ^T$ and Lemma \ref{lem:RfPPT} imply that $h(\ovQ)<h(Q)$.
\end{proof}

\begin{prop}\label{prop:degz2degx}
  Let $Q=(x,y,z)$ be an \PPT\ with $x \neq 0$ and  $\deg z= 2\deg x$. Then there exist $c\in K^*$ and  $f\in K[t]\backslash K$ such that
$Q=cS_f$, for $S_f$ as in \eqref{eq:St}.
\begin{proof}
We let $e=\deg x$ and $2e=\deg y=\deg z$. Using the euclidean algorithm, we can find $a\in K^*$ and $b\in K[t]$ satisfying $y=ax^2+b$ and $\deg b<2e$. And since $y$ and $z$ have the same leading coefficient, we have that $z=ax^2+\beta$, for some $\beta \in K[t]$ with $\deg \beta<2e$. From \eqref{eq:pyt}, we arrive at
\begin{equation}\label{eq:bbeta}
x^2(1+2a(b-\beta))=\beta^2-b^2. 
\end{equation}

If $b=0$ or $\beta=0$ then \eqref{eq:bbeta} implies that $\gcd(x,y,z)\neq 1$. Therefore, we assume that $\beta$ and $b$ are non-zero.

We show that $\beta^2-b^2= 0$. If we assume otherwise, then $(\beta-b)\neq 0$,  $(\beta+b)\neq 0$ and, by equating degrees in \eqref{eq:bbeta},
$$
2e=\deg(\beta+b)\leq \max\{\deg b,\deg \beta\}<2e.
$$
Also from \eqref{eq:bbeta} we  see that $\beta\neq b$, since $x\neq 0$. Therefore $b=-\beta$ and, consequently, \eqref{eq:bbeta} implies that
$$
1+4ab=0.
$$
This proves that
$$
Q=\left(x,ax^2-\frac{1}{4a},ax^2+\frac{1}{4a}\right).
$$
The result follows by taking $f=2ax$ and $c=1/4a$.
\end{proof}
\end{prop}

\begin{corollary}\label{cor:infinite_desc}
 Let $Q=(x,y,z)$ be an SPT with $x\neq 0$. 
Then there exist a  sequence $\{f,f_1, \dots f_k \}$ in $K[t] \backslash K$ and $c\in K^*$ such that
\[
Q^T = cM_{f_1} \cdots M_{f_k} S_f^T.
\]
\end{corollary}
\begin{proof}
If $Q=(x,y,z)$ satisfies $\deg z=2\deg x$ then $Q=cS_f$, where $c\in K^*$ and $f\in K[t]$ are given by
Proposition \ref{prop:degz2degx}. 

Thus we may assume that $\deg z\neq 2\deg x$. According to Proposition \ref{prop:descent}, there exists $f_1\in K[t]\backslash K$ such that $Q_1^T=(x_1,y_1,z_1)^T=M_{f_1}^{-1}Q^T$ is an \SPT\ with $x_1\neq 0$ and $h(Q_1)<h(Q)$. If $\deg z_1=2\deg x_1$ then, from Proposition \ref{prop:degz2degx} we find  $c\in K^*$ and $f\in K[t]$ such that
$$
Q^T=cM_{f_1}S_f^T,
$$
and we are done.  If $\deg z_1\neq 2\deg x_1$, then we can use Proposition \ref{prop:descent} to construct $Q_2^T=(x_2,y_2,z_2)^T=M_{f_2}^{-1}Q_1^T$, for some $f_2\in K[t]\backslash K$, such that $x_2\neq 0$ and 
$$
h(Q_2)<h(Q_1)<h(Q).
$$

Again, either $\deg z_2=2\deg x_2$ or $\deg z_2\neq2\deg x_2$. In the first case, we are finished because Proposition \ref{prop:degz2degx} implies
$$
Q^T=cM_{f_2}M_{f_1}S_f^T.
$$
Otherwise, we can use Proposition \ref{prop:descent} to construct an \SPT\ $Q_3^T=(x_3,y_3,z_3)^T=M_{f_3}^{-1}Q_2^T$ with $x_3\neq 0$ and 
$$
h(Q_3)<h(Q_2)<h(Q_1)<h(Q).
$$

In this fashion, for $Q_0=Q$, $i\geq 1$ and  some sequence  $\{f_1, \dots, f_i \}$ in $K[t] \backslash K$ we can construct a recursive sequence of \SPT's 
$$
(x_i,y_i,z_i)^T=Q_i^T=M_{f_i}^{-1}Q_{i-1}^T
$$
with $x_i\neq 0$ and
$$
h(Q_i)<\cdots <h(Q_2)<h(Q_1)<h(Q)
$$
Since $h(Q_i)$ is a positive integer for all $i$, the above inequality shows that we can not continue this construction indefinitely. Therefore, there exists an integer $k\geq 1$ such that $Q_k^T=(x_k,y_k,z_k)^T=M_{f_k}^{-1}Q_{k-1}^T$ is a \SPT\ with $x_k\neq 0$ and $\deg z_k=2\deg x_k$. Therefore, Proposition \ref{prop:degz2degx} guarantees the existence of $c\in K^*$ and $f\in K[t]$ such that
$$
Q_k=cS_f,
$$
 The result follows by noticing that
$$
cS_f^T=Q_k^T=M_{f_k}^{-1}Q_{k-1}^T=M_{f_{k}}^{-1}M_{f_{k-1}}^{-1}Q_{k-2}^T=\cdots =M_{f_{k}}^{-1}M_{f_{k-1}}^{-1}\cdots M_{f_1}^{-1}Q^T.
$$
\end{proof}

The next series of results are used to prove the uniqueness of the product representation
\[
Q^T = cM_{f_1} \cdots M_{f_k} S_f^T
\]
of any \SPT\ $Q$ with $x\neq 0$.

\begin{lemma}\label{lem:RpeqRq}
  Let $P$ and $Q$ be \PPT's and $g$ and $h$ be polynomials in $K[t]\backslash K$. If
  $$
  M_gP^T=M_hQ^T
  $$
  then $P=Q$ and $g=h$.
\end{lemma}
\begin{proof}
Write $f=g-h$. Clearly, it is enough to prove that $f=0$.
 
Use Lemma~\ref{prop:r+identity}, together with the fact that $R_f$ is of order 2
to obtain
%we arrive at
%By assumption, the definition of $M_t$, the identity $R_t=R_t^{-1}$ and 
\begin{equation*}%
U_1P^T=R_fR_0U_1Q^T. 
\end{equation*}
We rewrite this equation as 
\begin{equation}\label{eq:RgeqRh}
    \begin{pmatrix}
 \ovx,\ovy,\ovz
\end{pmatrix}^T=R_f
\begin{pmatrix}
 x,y,z
\end{pmatrix}^T
\end{equation}
where 
$(\ovx,\ovy,\ovz)^T=U_1P^T,
$ and 
$(x,y,z)^T=R_0U_1Q^T.
$

In what follows, we use notation from the proof of Lemma \ref{lem:NonPPTtoPPT}.
Since $P$ is an \PPT,  $(\ovx,\ovy,\ovz)$ is a Pythagorean triple of type (I). If we assume that $f\notin K$, Lemma \ref{lem:NonPPTtoPPT} would imply that the right-hand side of \eqref{eq:RgeqRh} is an \PPT. Therefore, $f\in K$.

Notice that $(x,y,z)$ is also a Pythagorean triple of type (I). Thus, if $f\in K^*$ then \eqref{eq:rplus_of_f} implies that $\deg x<\deg \ovx=\deg (z-y)=\deg z=\deg y$. Additionally, \eqref{eq:rplus_of_f} implies that $\ell(\ovy)=\ell(\ovx)f-\ell(z)$ and $\ell(\ovz)=\ell(\ovx)f+\ell(z)$. This shows that $\ell(\ovy)$ and $\ell(\ovz)$ cannot  be both zero; consequently, either $\deg \ovy=\deg z$ or $\deg \ovz=\deg z$. Therefore, $\deg \ovx=\deg \ovy$ or $\deg \ovx=\deg \ovz$. Since this contradicts the fact that $(\ovx,\ovy,\ovz)$ is a Pythagorean triple of type (I), we have that $f=g-h=0$, as desired.
\end{proof}

\begin{prop}\label{prop:equalproduct}
 Let $P$ and $Q$ be \PPT's and $m\geq n$ be positive integers. Suppose that   there are  sequences $\{g_0,g_1,\ldots,g_m\}$ and $\{h_0,h_1,\ldots,h_n\}$ in $K[t]\backslash K$ such that
 $$
 M_{g_m}M_{g_{m-1}}\cdots M_{g_1}P^T=M_{h_n}M_{h_{n-1}}\cdots M_{h_1} Q^T.
 $$
 Then either $P=Q$, $m=n$ and $g_i=h_i$; or $m>n$, $Q^T=M_{g_{m-n}}\cdots M_{g_1}P^T$ and $g_{m-n+i}=h_{i}$ for $1\leq i\leq n$.
\end{prop}
\begin{proof}
Assume first that $m=n$. We prove, by induction on $n$, that 
 $$
 M_{g_n}M_{g_{n-1}}\cdots M_{g_1}P^T=M_{h_n}M_{h_{n-1}}\cdots M_{h_1} Q^T
 $$
implies $P=Q$ and $g_i=h_i$, for all $1\leq i\leq n$. The base case of induction is Lemma \ref{lem:RpeqRq}. By Lemma \ref{lem:RfPPT}, $\bar{P}^T=M_{g_{n-1}}\cdots M_{g_1}P^T$ and $\bar{Q}^T=M_{h_{n-1}}\cdots M_{h_1} Q^T$ are \PPT's. By assumption, they satisfy
$$
M_{g_n}\bar{P}^T=M_{h_n}\bar{Q}^T.
$$
Another application of Lemma \ref{lem:RpeqRq} implies that $g_n=h_n$ and $\bar{P}=\bar{Q}$. Therefore, the induction hypothesis implies that $P=Q$ and $g_i=h_i$ for all $1\leq i\leq n$.

Suppose $m>n$.  If $P'^T=M_{g_{m-n}}M_{g_{j-1}}\cdots M_{g_1}P^T$ then, by hypothesis,
$$
 M_{g_m}M_{g_{m-1}}\cdots M_{g_{m-n+1}}{P'}^T=M_{h_n}M_{h_{n-1}}\cdots M_{h_1} Q^T.
$$
Notice that there are $n$ matrices in both sides of the previous equation. Therefore, we can apply the $m=n$ case which had been proved above to arrive at our desired result.
\end{proof}

\begin{lemma}\label{lem:SgP}
 Let $P$ be an \SPT, $c\in K^*$ and $f$ and $g$ be polynomials in $K[t]\backslash K$. Then
 $$
cS_g^T=M_fP^T
 $$
if and only if $g=2f$ and $P=(0,c,c)$.
\end{lemma}
\begin{proof}
We write $P=(x,y,z)$. From $cS_g^T=M_fP^T$ and \eqref{eq:rminus_of_f}, we get
\begin{eqnarray*}
2cg&=& -x +2 f (y+z)\\
cg^2-c&=&-2 f x  +2 f^{2}(y+z) -  y \\
cg^2+c&=&-2 f x  +2 f^{2}(y+z) +  z
 \end{eqnarray*}
Therefore,
$$
2c=cg^2+c-(cg^2-c)=z+y.
$$
Since $P$ is an \SPT, we conclude that $z=y=c$ and $x=0$. Consequently, the first equality above  implies that $g=2f$.

The converse is proved via direct computation using \eqref{eq:rminus_of_f}.
\end{proof}

\begin{corollary}\label{cor:uniqueness}
  Let  $m$ and $n$  be positive integers. Suppose that there are  $c,d\in K^*$ and sequences $\{g_0,g_1,\ldots,g_m\}$ and $\{h_0,h_1,\ldots,h_n\}$ in $K[t]\backslash K$ such that
 $$
c M_{g_m}M_{g_{m-1}}\cdots M_{g_1}S_{g_0}^T=dM_{h_n}M_{h_{n-1}}\cdots M_{h_1} S_{h_0}^T.
 $$
 Then $m=n$, $c=d$ and $g_i=h_i$, for all $0\leq i\leq n$.
\end{corollary}
\begin{proof}
% Assume first that $m=n$. We prove, by induction on $n$, that 
% $$
% M_{g_n}M_{g_{n-1}}\cdots M_{g_1}S_{g_0}^T=M_{h_n}M_{h_{n-1}}\cdots M_{h_1} S_{h_0}^T
% $$
%implies $S_{g_0}=S_{h_0}$ and $g_i=h_i$, for all $1\leq i\leq n$. The base case of induction is Lemma \ref{lem:RpeqRq}. By Lemma \ref{lem:RfPPT}, $\bar{P}=M_{g_{n-1}}\cdots M_{g_1}S_{g_0}^T$ and $\bar{Q}=M_{h_{n-1}}\cdots M_{h_1} S_{h_0}^T$ are \PPT's. By assumption, they satisfy
%$$
%M_{g_n}\bar{P}=M_{h_n}\bar{Q}.
%$$
%Another application of Lemma \ref{lem:RpeqRq} implies that $g_n=h_n$ and $\bar{P}=\bar{Q}$. Therefore, the induction hypothesis implies that $S_{g_0}=S_{h_0}$ and $g_i=h_i$, for all $1\leq i\leq n$.

First, we show that $m\neq n$ is impossible. Otherwise, we may assume without loss of generality that $m>n$ and conclude from Proposition \ref{prop:equalproduct} that 
\begin{equation}\label{eq:hoP}
cS_{h_0}^T=M_h\bar{P}^T 
\end{equation}
where $h=g_{m-n}$, and $\bar{P}=dS_{g_0}$ or $\bar{P}^T=dM_{g_{m-n-1}}\cdots M_{g_1}S_{g_0}^T$. In any case,  $\bar{P}=(\bar{x},\bar{y},\bar{z})$ is an \SPT\, with $\bar{x}\neq 0$, according to Lemma \ref{lem:RfPPT}. But, given \eqref{eq:hoP}, $\bar{x}\neq 0$ contradicts the conclusion of Lemma \ref{lem:SgP}. 

Therefore, $m=n$ and Proposition  \ref{prop:equalproduct} implies that $g_i=h_i$, for all $1\leq i\leq n$, and $cS_{g_0}=dS_{h_0}$. From the last equality, it easily follows  that $c=d$ and $g_0=h_0$, finishing our proof.
\end{proof}

%\begin{proof}[Proof of Theorem~\ref{thm:main_theorem_tree}]
%\end{proof}

\section{Generators of $O_{\Q}(K[t])$}
\label{sec:Corolaries}
When $(\ovx, \ovy, \ovz)^T = R_f (x, y, z)^T$, recall from \eqref{eq:rplus_of_f} that
\begin{equation}\label{eq:rplus_of_f2}
\begin{aligned}
 \ovx  &=   -x +2 f (z - y),\\
 \ovy&=(\ovx-x)f+y, \\
 \ovz&=(\ovx-x)f+z.
\end{aligned}
\end{equation}
We will use these identities in the following two lemmas.

\begin{lemma}\label{lem:SPT_Rf}
Suppose that $Q$ is an SPT. Then $R_fQ^T$ is also an SPT for any $f\in K$. 
\end{lemma}
\begin{proof}
Write $Q = (x, y, z)$ and $\tilde{Q}^T  = (\tilde{x}, \tilde{y}, \tilde{z})^T = R_fQ^T$ for $f\in K$.
Note from Lemma~\ref{lem:map_over_K_preserve_height} that $h(Q) = h(\tilde{Q})$.
From the first equation  in \eqref{eq:rplus_of_f2}, we see that that $\deg\ovx < \deg(y) = h(Q) = h(\tilde{Q})$.
This implies that $\deg\ovy = \deg\ovz$.
Also it is obvious from the second and third equations of \eqref{eq:rplus_of_f2} that $\ell(\ovy) = \ell(\ovz)$.
This completes the proof that $\tilde{Q}$ is an SPT.
\end{proof}
\begin{lemma}\label{lem:non_SPT_Rf}
Suppose that $Q$ is a primitive Pythagorean triple, which is not an SPT. 
Then there exists $f\in K$ such that $M_f^{-1}Q^T$
%\todo{RC: Isn't the displayed expression the same as $M_f^{-1}Q$?}
is an SPT.
\end{lemma}
\begin{proof}
Recall from Lemma~\ref{lem:non_SPT} that $Q$ is one of the type (I)--(IV).
We claim that there exists $f\in K$ such that $\tilde{Q}^T = (\ovx, \ovy, \ovz)^T := R_f Q^T$ is of type (I). 
This would imply the conclusion in the lemma because $U_1 \tilde{Q}^T = M_f^{-1} Q^T$ would then be an SPT.
Choose $f\in K$ so that
\begin{equation}\label{eq:condition_for_constant_f}
\begin{cases}
f= 0 & \text{ if $Q$ is of type (I)}, \\
-\ell(x) +2f\ell(z) = 0 & \text{ if $Q$ is of type (II)}, \\
-\ell(x) +2f(\ell(z) - \ell(y)) = 0 & \text{ if $Q$ is of type (III)}, \\
-\ell(x) -2f\ell(y) = 0 & \text{ if $Q$ is of type (IV)}.
\end{cases}
\end{equation}
We see from the first equation of \eqref{eq:rplus_of_f2} that the above equations are solvable in $f$ in all cases and that $h(Q) = h(\tilde{Q})$ because of Lemma~\ref{lem:map_over_K_preserve_height}.
Once $f$ is chosen to satisfy \eqref{eq:condition_for_constant_f}, we have $\deg\ovx < h(\tilde{Q})$,
which results in $\deg\ovy = \deg\ovz$.
This implies that either $\ell(\ovy) = -\ell(\ovz)$ (in which case the claim is proven) or $\ell(\ovy) = \ell(\ovz)$.
However, if $\ell(\ovy)= \ell(\ovz)$, this means that $\tilde{Q}$ is an SPT.
This is a contradiction because 
Lemma~\ref{lem:SPT_Rf} would then imply that 
\[
R_f \tilde{Q}^T = R_f (R_f Q)^T = Q^T
\]
is also an SPT, which we had assumed not.
\end{proof}

%We are now ready to Theorem~\ref{thm:transitivity_action}.
\begin{proof}[Proof of Theorem~\ref{thm:transitivity_action}]
We will prove that every primitive Pythagorean triple $Q$ is in the $O_{\Q}(K[t])$-orbit of $(0, 1, 1)$.

We handle the case $h(Q) = 0$.
If $Q$ is an SPT, then $Q = (0, c, c)$ for some $c\in K^*$.
Recall that $T_c$ is a matrix defined by \eqref{eq:definition_T_c}.
It is easily verified that 
\begin{equation}\label{eq:Tc}
 T_c(0, 1, 1)^T = (0, c, c)^T,
\end{equation}
so that $Q = (0, c, c)$ is in the $O_{\Q}(K[t])$-orbit of $(0, 1, 1)$. 
If $Q$ is not an SPT, we apply Lemma~\ref{lem:non_SPT_Rf} to obtain $f\in K$ such that $M_f^{-1}Q^T$ is an SPT.
Since $h(M_f^{-1}Q^T) = 0$ (cf.~Lemma~\ref{lem:map_over_K_preserve_height}) we see from the above argument that $M_f^{-1}Q^T$ is in the orbit of $(0, 1, 1)$, which of course implies that $Q$ is as well.
It remains to consider the case $h(Q) > 0$. 
Then Theorem~\ref{thm:main_theorem_tree} shows that $Q$ is in the orbit of $cS_g$, for some $c\in K^*$ and $g\in K[t]\backslash K$. But Lemma \ref{lem:SgP} implies that $cS_g$ and, thus, $Q$ are in the orbit of $(0,c,c)$,  which has already been shown to be in the orbit of $(0, 1, 1)$. 
\end{proof}
\begin{prop}\label{thm:stabilizer_011}
The stabilizer of $(0,1, 1)$ in $O_{\Q}(K[t])$ is generated by the set
\[
\{ R_f \mid f\in  K[t] \}.
\]
\end{prop}
\begin{proof}
We follow Conrad's argument given in Appendix of \cite{Conrad}.
Let $\mathcal{S}_1$ be the group generated by $\{ R_f \mid f \in K[t]\}$.
A simple calculation shows that $R_f$ fixes $(0, 1, 1)$ for every $f\in K[t]$
therefore $\mathcal{S}_1$ is a subgroup of the stabilizer of $(0, 1, 1)$.
Conversely, let $R$ be in the stabilizer of $(0, 1, 1)$.
Then $R$ is of the form
\[
R = 
\begin{pmatrix}
a_1 & a_2 & -a_2 \\
a_3 & a_4 & 1-a_4 \\
a_5 & a_6 & 1-a_6 \\
\end{pmatrix}
\]
for some $a_1, \dots, a_6 \in  K[t]$.
Also, letting $J$ be the $3\times 3$ diagonal matrix whose diagonal entries are $(1, 1, -1)$, the condition that $R$ is orthogonal with respect to the quadratic form $\Q(t)$ is equivalent to
\[
R^T J R = J,
\]
which yields
\[
\begin{cases}
a_1^2 + a_3^2 - a_5^2 = 1,\\
a_2^2 + a_4^2 - a_6^2 = 1,\\
a_2^2 + (a_4 - 1)^2 - (a_6- 1)^2 = -1,
\end{cases}
\begin{cases}
a_1a_2 + a_3a_4 - a_5a_6 = 0,\\
-a_1a_2 - a_3(a_4 - 1) + a_5(a_6 - 1) = 0,\\
-a_2^2 - a_4(a_4-1) + a_6(a_6-1)= 0.
\end{cases}
\]
Solving them simultaneously, we obtain
\[
\begin{cases}
a_1^2 = 1,\\
a_3 = a_5 = -a_1a_2,\\
a_4 = 1 + a_6 = 1 - \frac{a_2^2}2.\\
\end{cases}
\]
As a result, we have
\[
R = 
\begin{pmatrix}
a_1 & a_2 & -a_2 \\
-a_1a_2 & 1 - \frac{a_2^2}2 & \frac{a_2^2}2 \\
-a_1a_2 & - \frac{a_2^2}2& 1 + \frac{a_2^2}2
\end{pmatrix}
\]
with $a_1 = 1$ or $a_1 = -1$.
Therefore, we have
\[
  R = R_{-a_2/2} \text{ or } R_{-a_2/2}U_3
\]
depending on $a_1 = 1$ or $a_1 = -1$, respectively.
Since $U_3 = R_0 \in \mathcal{S}_1$ this completes the proof.
\end{proof}

%\ref{thm:orthogonal_generator}.
\begin{proof}[Proof of Theorem~\ref{thm:orthogonal_generator}]
Let $\mathcal{S}_2$ be the subgroup of $O_{\Q}(K[t])$ generated by the set given in the statement of the theorem.
Since
\[
U_1
=
\begin{pmatrix}
1 & 0 & 0 \\
0 & -1 & 0 \\
0 & 0 & 1
\end{pmatrix}
=
P_{xy}
\begin{pmatrix}
-1 & 0 & 0 \\
0 & 1 & 0 \\
0 & 0 & 1
\end{pmatrix}
P_{xy}
=
P_{xy}R_0P_{xy}
\]
it follows that $M_h\in \mathcal{S}_2$ for any $h\in K[t]$ (see \eqref{eq:definition_M_f}).
Suppose $A\in O_{\Q}(K[t])$.
Let $Q^T = A (0, 1, 1)^T$, which is obviously a primitive Pythagorean triple.
Next, we find $N\in \mathcal{S}_2$ such that $NQ^T$ is an SPT; if $Q$ is already an SPT, then $N = I_3$ (the identity matrix), otherwise, we choose $N$ to be $M_f^{-1}$ as constructed in Lemma~\ref{lem:non_SPT_Rf}.
Now, we apply Theorem~\ref{thm:main_theorem_tree} and Lemma \ref{lem:SgP} to obtain 
\[
\begin{pmatrix}
 0 \\ c \\ c
\end{pmatrix}
=
M^{-1}_{g/2}M_{f_1}  \cdots M_{f_k} N Q^T = 
M^{-1}_{g/2}M_{f_1}  \cdots M_{f_k} N A
\begin{pmatrix}
 0 \\ 1 \\ 1
\end{pmatrix}
\]
for some $g,f_1,  \cdots, f_k \in K[t]$ and $c\in K^*$.
Therefore we conclude from \eqref{eq:Tc} that
\[
T_c^{-1} M_{f_1} \cdots M_{f_k} N A
\]
fixes $(0, 1, 1)$.
From Proposition~\ref{thm:stabilizer_011}, it follows that the above product belongs to 
$\mathcal{S}_1$ (thus to $\mathcal{S}_2$ as well), which then implies that $A \in \mathcal{S}_2$.
\end{proof}

%\section{Results we may need}
%
%\begin{lemma}\label{lem:primpyt}
% A primitive pythagorean triple $(x,y,z)$ in $K[t]$ has the form
% $$
% x=\pm c(m^2-n^2), \quad y=\pm 2cmn,\quad z=\pm c(m^2+n^2)
% $$
% or 
%  $$
% x=\pm 2cmn, \quad y=\pm c(m^2-n^2),\quad z=\pm c(m^2+n^2)
% $$
%where $c\in K^*$ and $m,n\in K[t]$ are relatively prime.
%\end{lemma}
%\begin{proof}
% Give reference or proof.
%\end{proof}
%
%
%Notice that $S(P(f))=(d,2d,2d)$, for some $d>0$. The next lemma shows that the converse is true.

\end{document}